\documentclass[11pt, a4paper]{amsart}
\usepackage{graphicx} 
\usepackage[margin=1in]{geometry}
\title[Arithmetic Kodaira--Spencer Class and Frobenius Liftings]{Arithmetic Kodaira--Spencer Class and Frobenius Liftings via Frobenius--Witt Cotangent Complex}
\author{Kanau Shimada}
\date{}

\address{Department of Mathematics, Institute of Science Tokyo, 2-12-1 Ookayama, Meguro, Tokyo 152-8551}
\email{shimada.k.aj@m.titech.ac.jp}

\usepackage{amsmath}
\usepackage{amssymb}
\usepackage{latexsym}
\usepackage{amsthm}
\usepackage{amsfonts}
\usepackage{mathtools}
\usepackage{stmaryrd}
\usepackage{tikz}
\usepackage{tikz-cd}

\theoremstyle{plain}
\newtheorem{theorem}{Theorem}[section]
\newtheorem{lemma}[theorem]{Lemma}
\newtheorem{prop}[theorem]{Proposition}
\newtheorem{coro}[theorem]{Corollary}

\newtheorem{conj}[theorem]{Conjecture}

\theoremstyle{definition}
\newtheorem{definition}[theorem]{Definition}
\newtheorem{remark}[theorem]{Remark}

\newtheorem{notation}[theorem]{Notation}

\usepackage{mathrsfs}
\usepackage{blindtext}
\usepackage{hyperref}


\begin{document}
\maketitle
\begin{abstract}
    For a flat $\mathbb{Z}_{(p)}$-scheme $X$, we introduce an obstruction class to the existence of a Frobenius lift on its reduction $X_1$ modulo $p^2$. This class is defined by replacing the cotangent complex in the classical construction of the Kodaira--Spencer class with its arithmetic analogue, the Frobenius--Witt cotangent complex. We further prove that this class coincides with the obstruction to Frobenius lifting defined by classical deformation theory. We also define its relative version and prove that this is an obstruction class to the existence of a Frobenius lift compatible with the given one on the base scheme. 
\end{abstract}
\tableofcontents
\section{Introduction}
The existence of a lifting of Frobenius to $\mathbb{Z}/p^2$, or more generally to the ring $W_2(k)$ of $2$-truncated Witt vectors of a perfect field $k$, imposes strong restrictions on the geometric structure of a smooth projective variety in positive characteristic. For example, the following is a special case of \cite[Theorem 3]{BTLM97}. 
\begin{theorem}[Bott vanishing]
    Let $X$ be a smooth projective variety over a perfect field $k$ of characteristic $p>0$. Assume $X$ has a flat lift $\widetilde X$ to $W_2(k)$ and $\widetilde X$ has a Frobenius lift. Let $L$ be an ample line bundle on $X$. Then\[H^n(X,\Omega_X^m\otimes _{\mathcal O_X}L)=0\] holds for all $n>0$ and $m\ge0$.
\end{theorem}
Moreover, the following conjecture is proposed in \cite[Conjecture 1]{AWZ21}. 
\begin{conj}
    Let $X$ be a smooth projective variety over an algebraically closed field $k$ of characteristic $p>0$. Assume $X$ has a flat lift $\widetilde X$ to $W_2(k)$ and $\widetilde X$ has a Frobenius lift. Then there is a finite \'etale Galois covering $\pi:Y\to X$ such that the Albanese morphism $Y\to \operatorname{Alb}(Y)$ is a toric fibration. 
\end{conj}

In view of the above, it is an important problem to determine when a flat $W_2(k)$-scheme admits a Frobenius lift. For example, \cite{Zda18} includes a study of an obstruction class to the existence of a Frobenius lifting. There, the obstruction class is constructed using the deformation theory of \cite{Ill71} and \cite{Ill72}. On the other hand, for a smooth $W_2(k)$-scheme $X_1$, \cite{DKRZB17} constructs a short exact sequence \[ 0\to \mathcal O_{X_0}\to \Omega_{X_1}^{1,\mathrm{tot}}\to F^*_{X_0}\Omega_{X_0}\to 0\] using the module of total $p$-differentials $\Omega_{X_1}^{1,\mathrm{tot}}$, an arithmetic analogue of the module of K\"ahler differentials, and introduces an obstruction class to Frobenius lifting as the corresponding element of the above exact sequence in the $\operatorname{Ext}^1$-group. It is then shown that this obstruction class agrees, up to sign, with the class appearing in \cite{DI87}. Since the construction of this class is parallel to that of the classical Kodaira--Spencer class, it is called the arithmetic Kodaira--Spencer class. 

To extend this construction to the non-smooth case, it is necessary to extend the short exact sequence above to this setting. However, as in the case of the classical module of K\"ahler differentials, the resulting sequence need not be left exact. Therefore, in this paper, we extend the definition of the arithmetic Kodaira--Spencer class by replacing the module of total $p$-differentials $\Omega_{X_1}^{1,\mathrm{tot}}$ in the construction above with the Frobenius--Witt cotangent complex introduced in \cite{Shi26}: 
\begin{definition}
    Let $X$ be a flat $\mathbb Z_{(p)}$-scheme. Write $X_n:=X\times_{\operatorname{Spec}\mathbb Z_{(p)}}\operatorname{Spec}\mathbb Z/p^{n+1}$ for $n\ge0$. Consider the fiber sequence
    \[\mathcal O_{X_0}\to F\mathbb L_{X}\to F^*_{X_0}\mathbb L_{X_0/\mathbb F_p}\to.\] We define the arithmetic Kodaira--Spencer class $\kappa_X\in \operatorname{Ext}^1_{X_0}(F_{X_0}^*\mathbb L_{X_0/\mathbb F_p},\mathcal O_{X_0})$ by the class determined by the connecting map\[F_{X_0}^*\mathbb L_{X_0/\mathbb F_p}\to \mathcal O_{X_0}[1]\] of the above fiber sequence.
\end{definition}
We then prove that this class is equal to the obstruction class defined using the classical deformation theory considered in \cite{Zda18}. In the affine case (or more generally, animated $\mathbb Z_{(p)}$-algebras), we prove that such a coincidence holds for not just as classes in the $\operatorname{Ext}$-group but also for maps in the derived $\infty$-category (for a precise statement, see Theorem \ref{thm5.4}). As an application, we deduce that the Frobenius--Witt cotangent complex of a flat $\mathbb{Z}_{(p)}$-algebra essentially depends only on its reduction modulo $p^2$ (Corollary \ref{cor5.8}).

We also define the relative version of the arithmetic Kodaira--Spencer class:
\begin{definition}
    Let $f:X\to Y$ be a morphism of flat $\mathbb Z_{(p)}$-schemes. Let $f_n:X_n\to Y_n$ denote the pullback of $f$ along the canonical morphism $\operatorname{Spec}\mathbb Z/p^{n+1}\to \operatorname{Spec}\mathbb Z$ for non-negative integers $n\ge0$. The relative arithmetic Kodaira--Spencer class $\kappa_{X/Y}\in \operatorname{Ext}_{X_0}^1(F_{X_0}^*\mathbb L_{X_0/Y_0},f_0^*F\mathbb L_Y)$ is the element defined as the connecting map of the fiber sequence\[f_0^*F\mathbb L_Y\to F\mathbb L_X\to F_{X_0}^*\mathbb L_{X_0/Y_0}\to. \] 
\end{definition}
Furthermore, we prove that this class defines an obstruction class to the existence of a Frobenius lift compatible with the given one on the base scheme:
\begin{theorem}[Theorem \ref{thm3.9}]
    Let $f:X\to Y$ be a morphism of flat $\mathbb Z_{(p)}$-schemes. Let $f_n:X_n\to Y_n$ denote the pullback of $f$ along the canonical morphism $\operatorname{Spec}\mathbb Z/p^{n+1}\to \operatorname{Spec}\mathbb Z$ for $n\ge0$. Assume that there is a Frobenius lift on $Y_1$ with corresponding retraction $r:F\mathbb L_Y\to \mathcal O_{Y_0}$ (Remark \ref{remA.1}). Then the following assertions are equivalent.
    \begin{enumerate}
        \item The composition\[F_{X_0}^*\mathbb L_{X_0/Y_0}\xrightarrow[]{\kappa_{X/Y}}f_0^*F\mathbb L_Y[1]\xrightarrow[]{f_0^*r}f_0^*\mathcal O_{Y_0}[1]\simeq \mathcal O_{X_0}[1]\] is zero,
        \item there is a Frobenius lift on $X_1$ which is compatible with the given one on $Y_1$.
    \end{enumerate}
\end{theorem}

Finally, we mention a related work. During the final stages of this research, a preprint \cite{Mao26} by Zhouhang Mao was posted, which provides an extension of the definition of Frobenius--Witt cotangent complex to derived rings. The definition in that preprint is based on a different perspective from \cite{Shi26}. Although we are interested in whether the same results as in the present paper hold for the case of derived rings, and whether an alternative proof can be given using the  perspective of \cite{Mao26}, we do not address these questions here.

\vspace{2mm}
\textbf{Notation}.
We denote by $\operatorname{Ring}$ the category of (non-derived) rings. For a ring $A\in \operatorname{Ring}$, we denote by $\operatorname{Mod}(A)$ the category of (non-derived) $A$-modules. For a category $\mathcal C$ having a set of compact projective generators, we denote by $\operatorname{Ani}(\mathcal C)$ the animation of the category $\mathcal C$ (as defined in \cite[Section 5.1]{CS24}). For $A\in \operatorname{Ani(Ring})$, we denote by $\mathcal D(A)$ the derived $\infty$-category of $A$ and $\mathcal D_{\ge 0}(A)$ its connective part.

\vspace{2mm}
\textbf{Acknowledgments}. The author would like to express his sincere gratitude to Shou Yoshikawa. Corollary \ref{cor5.8} arose from a question posed by him, and the implication from (2) to (1) in Theorem \ref{thm3.2} was suggested by him during a fruitful discussion. The author is also grateful to Rirai Ikeda, Ryo Ishizuka and Ryoma Takeuchi for reading an earlier draft of this paper and providing valuable comments. Finally, the author would like to thank his supervisor, Yuichiro Taguchi, for his constant encouragement and guidance.

\section{Frobenius--Witt cotangent complexes for schemes}
The notion of a Frobenius--Witt cotangent complex was defined in \cite{Shi26} for animated $\mathbb Z_{(p)}$-algebras. In this section, we introduce the Frobenius–Witt cotangent complex for flat $\mathbb Z_{(p)}$-schemes and show that some of the basic properties proved in \cite{Shi26} also hold in this case. 

\begin{notation}
    For a scheme $X$, we define the $\infty$-category of quasi-coherent sheaves on $X$ by the following limit of $\infty$-categories: \[\mathcal D_{qc}(X):=\operatorname{lim}_{U\subseteq X}\mathcal D(\mathcal O_X(U)),\] where $U$ runs through all affine open subschemes of $X$, the transition maps are scalar extensions and $\mathcal D(\mathcal O_X(U))$ is the derived $\infty$-category of the ring $\mathcal O_X(U)$. By \cite[Theorem 1.1.4.4]{Lur17}, the $\infty$-category $\mathcal D_{qc}(X)$ is stable.
\end{notation}

\begin{definition}\label{Def1.2}
    Let $X$ be a flat $\mathbb Z_{(p)}$-scheme. Write $X_0:=X\times_{\operatorname{Spec}{\mathbb Z_{(p)}}}\operatorname{Spec}\mathbb F_p$. We define the Frobenius--Witt cotangent complex $F\mathbb L_X\in \mathcal D_{qc}(X_0)$ by associating the Frobenius--Witt cotangent complex $F\mathbb L_{\mathcal O_X(U)}$ to each affine open subscheme $U$ in $X$.
\end{definition}
\begin{remark}
    The well-definedness of the above definition follows from the functoriality of the correspondence\[A\mapsto F\mathbb L_A,\] (\cite[Proposition 4.10]{Shi26}) and the Zariski localization (\cite[Corollary 5.4]{Shi26}).
\end{remark}
The following is the Frobenius--Witt analogue of the transitivity fiber sequence for cotangent complexes of schemes, and it plays a crucial role in this paper.
\begin{theorem}\label{thm1.4}
    Let $f:X\to Y$ be a map of flat $\mathbb Z_{(p)}$-schemes. Let $f_0:X_0\to Y_0$ be the pullback of $f:X\to Y$ along $\operatorname{Spec}\mathbb F_p\to \operatorname{Spec}\mathbb Z_{(p)}$. Then there is a canonical fiber sequence
    \[f_0^*F\mathbb L_Y\to F\mathbb L_X\to F^*_{X_0}\mathbb L_{X_0/Y_0}\to \] in $\mathcal D_{qc}(X_0)$, where $f_0^*$ is the derived pullback along $f_0$ and $F_{X_0}^*$ is the derived pullback along the Frobenius map $F_{X_0}$ on $X_0$.
\end{theorem}
\begin{proof}
    This follows from the affine case established in \cite[Theorem~5.3]{Shi26}, its naturality and the commutativity of Frobenius maps with Zariski localization maps.
\end{proof}
\begin{coro}\label{cor1.5}
    Let $X$ be a flat $\mathbb Z_{(p)}$-scheme and $X_0$ its pullback along $\operatorname{Spec}\mathbb F_p\to \operatorname{Spec}\mathbb Z_{(p)}$. Then there is a canonical fiber sequence \[\mathcal O_{X_0}\to F\mathbb L_X\to F_{X_0}^*\mathbb L_{X_0/\mathbb F_p}\to \] in $\mathcal D_{qc}(X_0)$.
\end{coro}
\begin{proof}
    This follows by applying Theorem \ref{thm1.4} to the canonical map $f:X\to \operatorname{Spec}\mathbb Z_{(p)}$.
\end{proof}
\begin{remark}\label{rem1.6}
    In the situation of Corollary \ref{cor1.5}, denote by $s$ the first map\[\mathcal O_{X_0}\to F\mathbb L_{X}\] in the fiber sequence of Corollary \ref{cor1.5}. Then the composition\[\mathcal O_{X_0}\xrightarrow[]{s}F\mathbb L_X\xrightarrow[]{\mathrm{can}}F\Omega_X(=H_0(F\mathbb L_X))\] sends the identity element $1\in \mathcal O_{X_0}(X_0)$ to the element $w(p)\in H^0(X,F\Omega_X)$, where $w$ is the Frobenius--Witt differential (\cite[Definition 2.2]{Sai22}). This follows from \cite[Corollary 5.8]{Shi26}.
\end{remark}

At the end of this preliminary section, we will prove the following statement, which will be used in the proof of the main theorem in Section \ref{sec5}. This proposition means that the lower part of the long exact sequence of the fiber sequence in \cite[Theorem 5.3]{Shi26} is compatible with the one studied in \cite[Proposition 2.3]{Sai22}. 
\begin{prop}\label{pro2.7}
    Let $A\twoheadrightarrow A/I=:B$ be a surjective ring map of $\mathbb Z_{(p)}$-algebras with kernel $I$. Then the $B/p$-linear map \[F_{B/p}^*(I/I^2\otimes_BB/p)\to F\Omega_A\otimes_AB\] induced by the fiber sequence\[F\mathbb L_A\otimes_A^LB\to F\mathbb L_B\to F_{B/^Lp}^*\mathbb (L_{B/A}\otimes_B^LB/^Lp)\to \] in \cite[Theorem 5.3]{Shi26} can be written as \[F_{B/p}^*(I/I^2\otimes_BB/p)\to F\Omega_A\otimes_AB:\overline x\mapsto w(x)\otimes 1,\] where $\overline x\in I/I^2$ is the image of an element $x\in I$ and $w$ is the Frobenius--Witt differential of $A$.
\end{prop}
\begin{proof}
    For all $f\in I$, we have a map \[\mathbb Z_{(p)}[x]\to A:x\mapsto f,\] which induces a map $(x)/(x^2)\to I/I^2$. By the functoriality of the fiber sequence, the assertion reduces to the case \[A=\mathbb Z_{(p)}[x]\xrightarrow[]{x\mapsto 0}\mathbb Z_{(p)}.\] In this case, we will prove the statement by a direct computation. Let $Q_{\bullet}$ be the Bar resolution of $\mathbb Z_{(p)}$ as a $\mathbb Z_{(p)}[x]$-algebra. (For the definition, see \cite[Construction 4.13]{Iye07}). For the lower terms, we have $Q_0=\mathbb Z_{(p)}[x]$ and $Q_1=\mathbb Z_{(p)}[x,t]$.  Then we have the following diagram where the  horizontal sequences are exact:
    \[
\begin{tikzcd}
  0 \ar[r] & F\Omega_{\mathbb Z_{(p)}[x]}\otimes_{\mathbb Z_{(p)}[x]}Q_1 \ar[r] & F\Omega_{Q_1} \ar[r] & F^*_{Q_1/p}(\Omega_{Q_1/\mathbb Z_{(p)}[x]}\otimes_{Q_1}Q_1/p) \ar[r] & 0 \\
  0 \ar[r] & F\Omega_{\mathbb Z_{(p)}[x]}\otimes_{\mathbb Z_{(p)}[x]}Q_0 \ar[r] & F\Omega_{Q_0} \ar[r] & F^*_{Q_0/p}(\Omega_{Q_0/\mathbb Z_{(p)}[x]}\otimes_{Q_0}Q_0/p) \ar[r] & 0.
  \ar[from=1-2, to=2-2] 
  \ar[from=1-3, to=2-3] 
  \ar[from=1-4, to=2-4] 
\end{tikzcd}
\]
By the proof of \cite[Claim 3.30]{Bha12}, the element $dt\in \Omega_{Q_1/\mathbb Z_{(p)}[x]}$ is a representative of the element \[\overline x\in I/I^2=(x)/(x^2)\cong \pi_1(\mathbb L_{\mathbb Z_{(p)}/\mathbb Z_{(p)}[x]}).\] Furthermore, the element \[dt\otimes 1\in F^*_{Q_1/p}(\Omega_{Q_1/\mathbb Z_{(p)}[x]}\otimes_{Q_1}Q_1/p)\] is targeted by the element $w(t)\in F\Omega_{Q_1}$ under the upper right horizontal map in the above diagram. Moreover, the element $w(t)\in F\Omega_{Q_1}$ is sent to the element $w(x)\in F\Omega_{Q_0}$, which is targeted by the element $w(x)\otimes 1\in F\Omega_{\mathbb Z_{(p)}[x]}\otimes _{\mathbb Z_{(p)}[x]}Q_0$. This completes the proof.
\end{proof}
\begin{remark}
    This proposition slightly simplifies the proof of \cite[Corollary 6.2]{Shi26}.
\end{remark}

\section{Arithmetic Kodaira--Spencer classes and Frobenius lifts}\label{sec3}
In this section, we define the notion of arithmetic Kodaira--Spencer class for modulo $p^2$-reductions of flat $\mathbb Z_{(p)}$-schemes. Then we prove that the vanishing of this class is equivalent to the existence of Frobenius lifts.

\begin{definition}
    Let $X$ be a flat $\mathbb Z_{(p)}$-scheme. Write $X_0:=X\times _{\operatorname{Spec}\mathbb Z_{(p)}}\operatorname{Spec}\mathbb F_p$. The arithmetic Kodaira--Spencer class $\kappa_{X}$ for $X$ is the element $\kappa_X\in \operatorname{Ext}^1_{X_0}(F^*_{X_0}\mathbb L_{X_0/\mathbb F_p},\mathcal O_{X_0})$ defined by the connecting map $F_{X_0}^*\mathbb L_{X_0/\mathbb F_p}\to \mathcal O_{X_0}[1]$ of the fiber sequence\[\mathcal O_{X_0}\to F\mathbb L_X\to F_{X_0}^*\mathbb L_{X_0/\mathbb F_p}\to \] in Corollary \ref{cor1.5}.
\end{definition}
We can also define the arithmetic Kodaira--Spencer map for animated rings.

\begin{definition}\label{def5.1}
    Let $A$ be an animated $\mathbb Z_{(p)}$-algebra. We define the arithmetic Kodaira--Spencer map \[\kappa_A:F^*_{A/^Lp}\mathbb L_{(A/^Lp)/\mathbb F_p}\to A/^Lp[1]\] by the connecting map of the following fiber sequence in $\mathcal D(A/^Lp)$ (\cite[Theorem 5.3]{Shi26}):\[A/^Lp\to F\mathbb L_A\to F^*_{A/^Lp}\mathbb L_{(A/^Lp)/\mathbb F_p}\to. \]
\end{definition}

\begin{theorem}\label{thm3.2}
    Let $X$ be a flat $\mathbb Z_{(p)}$-scheme. Write $X_0:=X\times _{\operatorname{Spec}\mathbb Z_{(p)}}\operatorname{Spec}\mathbb F_p$ and  $X_1:=X\times _{\operatorname{Spec}\mathbb Z_{(p)}}\operatorname{Spec}\mathbb Z/p^2$. Then the following assertions are equivalent.\begin{enumerate}
        \item The arithmetic Kodaira--Spencer class $\kappa_X$ of $X$ is zero.
        \item There is a Frobenius lift on $X_1$.
    \end{enumerate}
\end{theorem}
\begin{proof}
    $(1)\Rightarrow (2)$. Assume $\kappa_X=0$. Then the fiber sequence \[\mathcal O_{X_0}\xrightarrow[]{s} F\mathbb L_X\to F_{X_0}^*\mathbb L_{X_0/\mathbb F_p}\to \] has a splitting, so there is a map $r:F\mathbb L_X\to \mathcal O_{X_0}$ such that $r\circ s=\mathrm{id}$. Since $\mathcal O_{X_0}$ is concentrated in degree zero, the map $r$ induces a map $r':F\Omega_X\to \mathcal O_{X_0}$. By $r\circ s=\mathrm{id}$ and Remark \ref{rem1.6}, we have $r'(w(p))=1$. The composition \[\alpha:\mathcal O_X\xrightarrow[]{w}F\Omega_X\xrightarrow[]{r'}\mathcal O_{X_0},\] where we see $F\Omega_X$ and $\mathcal O_{X_0}$ as quasi-coherent sheaves on $X$ via the (non-derived) pushforward along the closed immersion $X_0\hookrightarrow X$,  satisfies $\alpha(x+p^2y)=\alpha(x)$ for any local sections $x, y\in \mathcal O_X$. Indeed this follows from the following direct computation:\begin{align*}
        \alpha(x+p^2y)&=r'(w(x+p^2y))\\&=r'(w(x))+r'(w(p^2y))-r'(P(x,p^2y)w(p))\\&=\alpha(x)+r'(w(p)(py)^p+p^pw(py))-P(x,p^2y)\\&=\alpha(x),
    \end{align*}
    where we used the axiom of Frobenius--Witt derivations (see \cite[Definition 1.1]{Sai22} for the definition of Frobenius--Witt derivations and the polynomial $P(X,Y)\in \mathbb Z[X,Y]$) in the second equality, $r'(w(p))=1$ in the third equality and $p=0\in \mathcal O_{X_0}$ in the final equality. Therefore, the map $\alpha$ induces a map $\alpha_1:\mathcal O_{X_1}\to \mathcal O_{X_0}$. We define the map $F:\mathcal O_{X_1}\to \mathcal O_{X_1}$ by $x\mapsto x^p+p\alpha_1(x)$. Then the following argument implies that this map is a Frobenius lift on $X_1$.
    \begin{itemize}
        \item The fact that $F$ is a ring map follows from the following computations: \begin{align*}
            F(0)&=0^p+p\alpha_1(0)\\&=p\alpha(0)\\&=pr'(w(0))\\&=0,\\
            F(1)&=1^p+p\alpha_1(1)\\&=1+p\alpha(1)\\&=1+pr'(w(1))\\&=1,\\
            F(\overline x+\overline y)&=(\overline x+\overline y)^p+pr'(w(x+y))\\&=(\overline x+\overline y)^p+pr'(w(x)+w(y)-P(x,y)w(p))\\&=(\overline x+\overline y)^p+p\alpha(x)+p\alpha(y)-pP(x,y)\\&=(\overline x^p+p\alpha_1(\overline x))+(\overline y^p+p\alpha_1(y))\\&=F(\overline x)+F(\overline y),\\
            F(\overline x\overline y)&=(\overline{xy})^p+pr'(w(xy))\\&=(\overline{xy})^p+pr'(w(x)y^p+w(y)x^p)\\&=(\overline{xy})^p+p\alpha_1(\overline x)y^p+p\alpha_1(\overline y)x^p\\&=(\overline x^p+p\alpha_1(\overline x))(\overline y^p+p\alpha_1(\overline y)),
        \end{align*}
        where $x$ and $y$ are any local sections of $\mathcal O_X$ and $\overline x$ and $\overline y$ are their reductions in $\mathcal O_{X_1}$. 
        \item The fact that $F$ is a Frobenius lift follows from the definition of $F$.
    \end{itemize}
$(2)\Rightarrow (1)$. Conversely, assume there is a Frobenius lift $F$ on $X_1$. We also denote by $F$ the endomorphism of the structure sheaf $\mathcal O_{X_1}$ induced by the given Frobenius lift on $X_1$. Then the image of the map \[F-(-)^p:\mathcal O_{X_1}\to \mathcal O_{X_1}\] is contained in $p\mathcal O_{X_1}/p^2\mathcal O_{X_1}$. Since $X_1$ is flat over $\mathbb Z/p^2$, we have \[\mathcal O_{X_0}\cong \mathcal O_{X_1}/p\mathcal O_{X_1}\xrightarrow[\sim]{\times p}p\mathcal O_{X_1}/p^2\mathcal O_{X_1}.\] Composing the above map with the inverse of this isomorphism, we get the following map:
\[\alpha:\mathcal O_X\xrightarrow[]{\mathrm{can}}\mathcal O_{X_1}\xrightarrow[]{F-(-)^p} pO_{X_1}/p^2\mathcal O_{X_1}\cong \mathcal O_{X_0}.\] The following computations imply that the map $\alpha$ is a Frobenius--Witt derivation.\begin{align*}
    p\alpha(x+y)&=F(x+y)-(x+y)^p\\&=(F(x)-x^p)+(F(y)-y^p)-((x+y)^p-x^p-y^p)\\&=p(\alpha(x)+\alpha(y)-P(x,y))\in p\mathcal O_{X_1}/p^2\mathcal O_{X_1},\\
    p\alpha(xy)&=F(xy)-(xy)^p\\&=F(x)F(y)-x^py^p\\&=(x^p+p\alpha(x))(y^p+p\alpha(y))-x^py^p\\&=p(\alpha(x)y^p+\alpha(y)x^p)\in p\mathcal O_{X_1}/p^2\mathcal O_{X_1}. 
\end{align*}
Thus the map $\alpha$ induces a $\mathcal O_{X_0}$-linear map \[h:F\Omega_{X}^1\to \mathcal O_{X_0}.\] Since we have \begin{align*}
    p\alpha(p)&=F(p)-p^p\\&=p-p^p\\&=p\in p\mathcal O_{X_1}/p^2\mathcal O_{X_1},
\end{align*} $\alpha(p)=1$ holds. Therefore, Remark \ref{rem1.6} ensures that the composition\[r:F\mathbb L_X\xrightarrow[]{\mathrm{can}}F\Omega_X\xrightarrow[]{h}\mathcal O_{X_0}\] is a retraction of the map $s:\mathcal O_{X_0}\to F\mathbb L_{X}$. This means $\kappa_X=0$.
\end{proof}

\begin{remark}\label{remA.1}
    In the situation of Theorem \ref{thm3.2}, the above proof implies that there is a natural bijection between the set of Frobenius lifts on $X_1$ and the set of retractions of the canonical map $\mathcal O_{X_0}\to F\mathbb L_X$.
\end{remark}

\begin{theorem}\label{thm3.3}
    Let $X$ be a flat $\mathbb Z_{(p)}$-scheme. Write $X_0:=X\times _{\operatorname{Spec}\mathbb Z_{(p)}}\operatorname{Spec}\mathbb F_p$ and  $X_1:=X\times _{\operatorname{Spec}\mathbb Z_{(p)}}\operatorname{Spec}\mathbb Z/p^2$. Assume that the scheme $X_1$ has a Frobenius lift. Then the set of Frobenius lifts on $X_1$ is a torsor under the group $\operatorname{Hom}_{X_0}(F_{X_0}^*\mathbb L_{X_0/\mathbb F_p},\mathcal O_{X_0})$.
\end{theorem}
\begin{proof}
    By Corollary \ref{cor1.5}, we have the following fiber sequence in the infinity category $\mathcal S$ of spaces :\[\operatorname{Hom}_{X_0}(F_{X_0}^*\mathbb L_{X_0/\mathbb F_p},\mathcal O_{X_0})\to \operatorname{Hom}_{X_0}(F_{X_0}^*\mathbb L_{X_0/\mathbb F_p},F\mathbb L_{X})\xrightarrow[]{\pi} \operatorname{Hom}_{X_0}(F_{X_0}^*\mathbb L_{X_0/\mathbb F_p},F_{X_0}^*\mathbb L_{X_0/\mathbb F_p}).\] By the proof of Theorem \ref{thm3.2}, the set of Frobenius lifts on $X_1$ is isomorphic to the space of splittings of the fiber sequence in Corollary \ref{cor1.5}, which is isomorphic to the fiber of $\pi$ over $\mathrm{id}\in \operatorname{Hom}_{X_0}(F_{X_0}^*\mathbb L_{X_0/\mathbb F_p},F_{X_0}^*\mathbb L_{X_0/\mathbb F_p})$.
\end{proof}
Next, we see an example of a computation of the Frobenius--Witt cotangent complex using Theorem \ref{thm3.2}:
\begin{prop}
    Let $X$ be an abelian scheme over $\mathbb Z_{(p)}$. Write $X_n:=X\times_{\operatorname{Spec}\mathbb Z_{(p)}}\operatorname{Spec}\mathbb Z/p^{n+1}$ for $n\ge0$. Then the following assertions are equivalent:\begin{enumerate}
        \item $X_1$ has a Frobenius lift,
        \item the Frobenius--Witt cotangent complex $F\mathbb L_X$ is a free $\mathcal O_{X_0}$-module.
    \end{enumerate}
\end{prop}
\begin{proof}
   $(1)\Rightarrow (2)$. By Theorem \ref{thm3.2}, the fiber sequence\[\mathcal O_{X_0}\to F\mathbb L_X\to F_{X_0}^*\mathbb L_{X_0/\mathbb F_p}\to \] splits. Since $X_0$ is an abelian variety over $\mathbb F_p$, the cotangent complex $\mathbb L_{X_0/\mathbb F_p}$ is a free $\mathcal O_{X_0}$-module. Combining them, we see that the Frobenius--Witt cotangent complex $F\mathbb L_X$ is a free $\mathcal O_{X_0}$-module.

   $(2)\Rightarrow (1)$. Assume that the Frobenius--Witt cotangent complex $F\mathbb L_X$ is a free $\mathcal O_{X_0}$-module. Then the fiber sequence\[\mathcal O_{X_0}\to F\mathbb L_X\to F_{X_0}^*\mathbb L_{X_0/\mathbb F_p}\to\] is a short exact sequence of finite free $\mathcal O_{X_0}$-modules. Take an isomorphism $F\mathbb L_X\simeq \mathcal O_{X_0}^{\oplus m}$. Then the first map $\mathcal O_{X_0}\to F\mathbb L_X$ in the above fiber sequence can be identifies with the $\mathcal O_{X_0}$-linear injection $\mathcal O_{X_0}\hookrightarrow \mathcal O_{X_0}^{\oplus m}$. Since $X_0$ is geometrically integral and proper over $\mathbb F_p$, this map can be written as \[f\mapsto (a_1f,a_2f,\dots,a_mf)\] for any local section $f\in \mathcal O_{X_0}$ using some elements $a_1,\dots,a_m\in \mathbb F_p$. Since this map is injective, there is some $i\in \{1,\dots,m\}$ such that $a_i\neq0\in \mathbb F_p$. Using such $i$, we can define an $\mathcal O_{X_0}$-linear map\[\mathcal O_{X_0}^{\oplus m}\to \mathcal O_{X_0}:(g_1,\dots,g_m)\mapsto g_i/a_i,\] which is a retraction of the map $\mathcal O_{X_0}\to \mathcal O_{X_0}^{\oplus m}$. Therefore, Theorem \ref{thm3.2} implies the assertion $(1)$.
\end{proof}

At the end of this section, we define a relative version of arithmetic Kodaira--Spencer class and prove its properties.

\begin{definition}
    Let $f:X\to Y$ be a morphism of flat $\mathbb Z_{(p)}$-schemes. Let $f_n:X_n\to Y_n$ denote the pullback of $f$ along the canonical morphism $\operatorname{Spec}\mathbb Z/p^{n+1}\to \operatorname{Spec}\mathbb Z$ for non-negative integers $n\ge0$. The relative arithmetic Kodaira--Spencer class $\kappa_{X/Y}\in \operatorname{Ext}_{X_0}^1(F_{X_0}^*\mathbb L_{X_0/Y_0},f_0^*F\mathbb L_Y)$ is the element defined as the connecting map of the fiber sequence\[f_0^*F\mathbb L_Y\to F\mathbb L_X\to F_{X_0}^*\mathbb L_{X_0/Y_0}\to \] of Theorem \ref{thm1.4}.
\end{definition}
We can also define the relative arithmetic Kodaira--Spencer map for maps of animated rings:
\begin{definition}
    Let $A\to B$ be a map of animated $\mathbb Z_{(p)}$-algebras. The relative arithmetic Kodaira--Spencer map $\kappa_{B/A}:F_{B/^Lp}^*\mathbb L_{(B/^Lp)/(A/^Lp)}\to F\mathbb L_A\otimes_A^LB[1]$ is defined by the connecting map of the fiber sequence:\[F\mathbb L_A\otimes_A^LB\to F\mathbb L_B\to F_{B/^Lp}^*\mathbb L_{(B/^Lp)/(A/^Lp)}\to \] of \cite[Theorem 5.3]{Shi26}.
\end{definition}
The following assertion is a relative version of Theorem \ref{thm3.2}.
\begin{theorem}\label{thm3.9}
    Let $f:X\to Y$ be a morphism of flat $\mathbb Z_{(p)}$-schemes. Let $f_n:X_n\to Y_n$ denote the pullback of $f$ along the canonical morphism $\operatorname{Spec}\mathbb Z/p^{n+1}\to \operatorname{Spec}\mathbb Z$ for $n\ge0$. Assume that there is a Frobenius lift on $Y_1$ with corresponding retraction $r:F\mathbb L_Y\to \mathcal O_{Y_0}$ (Remark \ref{remA.1}). Then the following assertions are equivalent.
    \begin{enumerate}
        \item The composition\[F_{X_0}^*\mathbb L_{X_0/Y_0}\xrightarrow[]{\kappa_{X/Y}}f_0^*F\mathbb L_Y[1]\xrightarrow[]{f_0^*r}f_0^*\mathcal O_{Y_0}[1]\simeq \mathcal O_{X_0}[1]\] is zero,
        \item there is a Frobenius lift on $X_1$ which is compatible with the given one on $Y_1$.
    \end{enumerate}
\end{theorem}
\begin{proof}
$(1)\Rightarrow (2)$. Assume that the composition $f_0^*r\circ \kappa_{X/Y}$ is zero. By the universal property of pushout, we get the following commutative diagram:
\[
\begin{tikzcd}
  F^*_{X_0}\mathbb L_{X_0/Y_0} \lbrack -1 \rbrack \arrow[r, "\kappa_{X/Y}\lbrack -1\rbrack "] \arrow[d,] 
    & f^*_0F\mathbb L_{Y} \arrow[d, ] \arrow[ddr, "f^*_0r", bend left=20] \\
  0 \arrow[r, ] \arrow[drr, bend right=20] 
    & F\mathbb L_X \arrow[dr, "q"] \\
  & & \mathcal O_{X_0}.
\end{tikzcd}
\]
The upper triangle in the above diagram yields the following commutative diagram:
\[
\begin{tikzcd}
  \mathcal O_{X_0} \arrow[d] \arrow[ddr, "\mathrm{id}", bend left=20] & \\
  f^*_0F\mathbb L_{Y} \arrow[d] \arrow[dr, "f^*_0r"] & \\
  F\mathbb L_X \arrow[r, "q"] & \mathcal O_{X_0},
\end{tikzcd}
\]
which implies that the map $q$ is a retraction of the canonical map $\mathcal O_{X_0}\to F\mathbb L_X$ and the corresponding Frobenius lift on $X_1$ is compatible with the one on $Y_1$ corresponding to $r$.

$(2)\Rightarrow (1)$. Assume that there is a Frobenius lift on $X_1$ which is compatible with the given one on $Y_1$. Let $q$ be the corresponding retraction of the canonical map $ \mathcal O_{X_0}\to F\mathbb L_X$. Since the Frobenius lift under consideration on $X_1$ is compatible with the given one on $Y_1$, we have the following commutative diagram:

\[\begin{tikzcd}
    f_0^*F\mathbb L_Y \arrow[d] \arrow[dr, "f_0^*r"] &\\
    F\mathbb L_X \arrow[r, "q"] & \mathcal O_{X_0}.
\end{tikzcd}\]
Since the composition
\[ F^*_{X_0}\mathbb L_{X_0/Y_0} \lbrack -1 \rbrack \xrightarrow[]{\kappa_{X/Y}[-1]}f_0^*F\mathbb L_Y\to F\mathbb L_X\]
 is zero, the composition $f^*_0r\circ \kappa_{X/Y}$ is also zero.

\end{proof}

\section{Comparison with Deligne--Illusie classes}
In this section, we prove that the arithmetic Kodaira--Spencer class for smooth schemes coincides with the Deligne--Illusie class defined by the cocycle arising from the local Frobenius lifts. This concept goes back to the earlier work of Deligne--Illusie (\cite{DI87}), while the terminology was introduced in \cite{DKRZB17}.
\begin{definition}{\cite{DI87}}
    Let $X_1$ be a smooth $\mathbb Z/p^2$-scheme. Let $\{U_i\}_{i\in I}$ be an affine open covering of $X_1$ such that each $U_{i}$ has a Frobenius lift $F_i$ (such a covering always exists by the lifting property of smoothness). Then the difference \[F_i-F_j:\mathcal O_{X_1}(U_i\cap U_j)\to \mathcal O_{X_1}(U_i\cap U_j)\] induces the following map:\[\frac{1}{p}(F_i-F_j):\mathcal O_{X_1}(U_i\cap U_j)\to p\mathcal O_{X_1}(U_i\cap U_j)\xrightarrow[\sim]{\times 1/p}\mathcal O_{X_0}(U_i\cap U_j),\] which is a derivation. Thus this derivation induces the following $\mathcal O_{X_0}$-linear map:\[h_{ij}:\Omega_{X_0/\mathbb F_p}\to F_{X_0,*}\mathcal O_{X_0}.\] By construction, the collection $\{h_{ij}\}_{i,j\in I}$ satisfies the cocycle condition, so it defines a class of $H^1(X_0,F_{X_0}^*T_{X_0/\mathbb F_p})$, where $T_{X_0/\mathbb F_p}$ is the tangent sheaf of $X_0$ over $\mathbb F_p$. We call this class the Deligne--Illusie class of $X_1$ and denote it by $h_{X_1}$.
\end{definition}
\begin{theorem}
    Let $X$ be a smooth $\mathbb Z_{(p)}$-scheme. Write $X_1:=X\times _{\operatorname{Spec}\mathbb Z_{(p)}}\operatorname{Spec}\mathbb Z/p^2$ and $X_0:=X\times _{\operatorname{Spec}\mathbb Z_{(p)}}\operatorname{Spec}\mathbb F_p$. Then the arithmetic Kodaira--Spencer class of $X$ and the Deligne--Illusie class coincide up to sign:\[\kappa_X=-h_{X_1}\] under the identification \[\operatorname{Ext}^1_{X_0}(F^*_{X_0}\mathbb L_{X_0/\mathbb F_p},\mathcal O_{X_0})\cong H^1(X_0,F_{X_0}^*T_{X_0/\mathbb F_p}).\]
\end{theorem}
\begin{proof}
    The proof is similar to that of \cite[Theorem 3.2]{DKRZB17}. However, since our notation (such as $F\Omega_X$) is defined slightly differently, we include the proof here for completeness.

    Since the Frobenius--Witt cotangent complex $F\mathbb L_X$ and the cotangent complex $\mathbb L_{X_0/\mathbb F_p}$ are concentrated in degree zero in this case, the fiber sequence in Corollary \ref{cor1.5} coincides with the following short exact sequence:\[0\to \mathcal O_{X_0}\to F\Omega_X^1\xrightarrow[]{r} F_{X_0}^*\Omega_{X_0/\mathbb F_p}\to 0.\]
    Thus the arithmetic Kodaira--Spencer class $\kappa_X$ of $X$ coincides with the class in $\operatorname{Ext}^1_{X_0}(F^*_{X_0}\Omega_{X_0/\mathbb F_p},\mathcal O_{X_0})$ corresponding to the above short exact sequence. The image of this class under the canonical isomorphism\[\operatorname{Ext}^1_{X_0}(F^*_{X_0}\Omega_{X_0/\mathbb F_p},\mathcal O_{X_0})\cong \check H^1(X_0,\mathcal Hom(F^*_{X_0}\Omega_{X_0/\mathbb F_p},\mathcal O_{X_0}))\] can be written as follows. Let $\{U_i\}_{i\in I}$ be an affine open covering of $X$ with $U_{i,n}:=U_i\times_{\operatorname{Spec}\mathbb Z_{(p)}}\operatorname{Spec}\mathbb Z/p^{n+1}$ such that the map $F\Omega_{U_i}\to F^*_{U_{i,0}}\Omega_{U_{i,0}/\mathbb F_p}$ has a section $s_i$. Then the difference $s_i-s_j$ induces a map\[s_{ij}:=s_i-s_j:F^*_{U_{i,0}\cap U_{j,0}}\Omega_{U_{i,0}\cap U_{j,0}/\mathbb F_p}\to \mathcal O_{U_{i,0}\cap U_{j,0}}.\] The collection $\{s_{ij}\}_{i,j\in I}$ defines a class of $\check H^1(X_0,\mathcal Hom(F^*_{X_0}\Omega_{X_0/\mathbb F_p},\mathcal O_{X_0}))$, which is the desired one. On the other hand, we have a Frobenius lift $F_i$ on $U_{i,1}$ defined by \[F_i(x):=x^p+p(w(x)-s_irw(x)).\] Then we have\[(F_i(x)-F_j(x))/p=s_jrw(x)-s_irw(x)=s_j(dx)-s_i(dx)=-s_{ij}(dx),\] where $w$ is the Frobenius--Witt differential and $d$ is the K\"ahler differential.
    This implies \[h_{ij}=-s_{ij}\] for all $i,j\in I$, so we have $h_{X_1}=-\kappa_{X}$.  
\end{proof}

\section{Comparison with the classical deformation theory}\label{sec5}
In this section, we prove that the arithmetic Kodaira--Spencer class constructed in Section \ref{sec3} coincides with the obstruction class defined using the classical theory of the cotangent complex. 

First, we compare the arithmetic Kodaira--Spencer map for animated rings with the one defined by the classical deformation theory of cotangent complexes.

\begin{definition}\label{Def5.2}
Let $A$ be an animated $\mathbb Z_{(p)}$-algebra. We denote by $\xi:\mathbb L_{(A/^Lp)/(\mathbb Z/p^2)}\to A/^Lp[1]$ the $A/^Lp$-linear map classifying the following square-zero extension:\[A/^Lp\to A/^Lp^2\to A/^Lp\to .\] 
\end{definition}

\begin{definition}\label{Def5.3}
    Let $\lambda:F^*_{A/^Lp}\mathbb L_{(A/^Lp)/(\mathbb Z/p^2)}\to A/^Lp[1]$ be the map adjoint to the map $\xi\circ dF-F\circ \xi$ defined by the difference of the following two compositions:\[\xi\circ dF:L_{(A/^Lp)/(\mathbb Z/p^2)}\xrightarrow[]{dF_{A/^Lp}}L_{(A/^Lp)/(\mathbb Z/p^2)}\xrightarrow[]{\xi}A/^Lp[1],\]
    \[F\circ \xi:L_{(A/^Lp)/(\mathbb Z/p^2)}\xrightarrow[]{\xi}A/^Lp[1]\xrightarrow[]{F_{A/^Lp}[1]}A/^Lp[1].\]
\end{definition}
The main theorem in this section is the following:
\begin{theorem}\label{thm5.4}
    Let $A$ be an animated $\mathbb Z_{(p)}$-algebra. Then the map $\lambda:F^*_{A/^Lp}\mathbb L_{(A/^Lp)/(\mathbb Z/p^2)}\to A/^Lp[1]$ is canonically equivalent to the following composition:\[F^*_{A/^Lp}\mathbb L_{(A/^Lp)/(\mathbb Z/p^2)}\xrightarrow[]{\mathrm{can}} F^*_{A/^Lp}\mathbb L_{(A/^Lp)/(\mathbb Z/p)}\xrightarrow[]{\kappa_A} A/^Lp[1].\]
\end{theorem}
\begin{proof}
    Since all terms in the maps under consideration commute with sifted colimits on $A\in \operatorname{Ani(Ring)}_{\mathbb Z_{(p)}/}$ and all maps are functorial, we may assume $A$ is a finitely generated polynomial $\mathbb Z_{(p)}$-algebra $\mathbb Z_{(p)}[\underline{X}]$.

    Consider the canonical map \[A:=\mathbb Z_{(p)}[\underline{X}]\to \mathbb Z_{(p)}[\underline{X}^{1/p^{\infty}}]=:B\]
    Let $B^{\bullet}$ be the \v{C}ech conerve of the map $A\to B$. Since each term in the maps under consideration satisfies fpqc descent (see Lemma \ref{lem5.5}), the assertion reduces to the case $A=B_n$ ($n\ge0$). 

    Therefore, we may assume $A=\mathbb Z_{(p)}[\underline{X}^{1/p^{\infty}}]/I$, where $I$ is generated by a Koszul regular sequence in $\mathbb Z_{(p)}[\underline{X}^{1/p^{\infty}}]$. In this case, the cotangent complex $\mathbb L_{(A/p)/(\mathbb Z/p^2)}$ is concentrated in homological degree $\ge 1$, so it suffices to check that the maps \[\kappa_A\circ \mathrm{can},\,\lambda:F_{A/p}^*\mathbb L_{(A/p)/(\mathbb Z/p^2)}\to A/p[1]\] induce the same map $F_{A/p}^*\pi_1(\mathbb L_{(A/p)/(\mathbb Z/p^2)})\to A/p$ between classical $A/p$-modules.

    First, we will compute the map $\kappa_A$. Consider the following diagram of fiber sequences in $\mathcal D(A/p)$:\[
\begin{tikzcd}
F\mathbb L_{\mathbb Z_{(p)}}\otimes^L_{\mathbb F_p}A/p \arrow[r, ] \arrow[d, ] 
    & F\mathbb L_A \arrow[r, ] \arrow[d, ] 
    & F_{A/p}^*\mathbb L_{(A/p)/\mathbb F_p}  \arrow[d, ]\\
F\mathbb L_{\mathbb Z_{(p)}[\underline{X}^{1/p^{\infty}}]}\otimes^L_{\mathbb F_p[\underline{X}^{1/p^{\infty}}]}A/p \arrow[r, ] 
    & F\mathbb L_A \arrow[r, ] 
    &  F_{A/p}^*\mathbb L_{(A/p)/\mathbb F_p[\underline{X}^{1/p^{\infty}}]} 
\end{tikzcd}
\]
Since $\kappa_A$ is the connecting map of the upper fiber sequence and the vertical maps are equivalences in the above diagram, $\kappa_A$ can be computed as the connecting map of the lower fiber sequence. Thus, Proposition \ref{pro2.7} ensures that the map $\kappa_A$ can be identified with the following map:\[F_{A/p}^*(\overline{I}/\overline{I}^2)\xrightarrow[]{\overline {x}\otimes 1\mapsto w(x)\otimes 1}F\Omega_{\mathbb Z_{(p)}[\underline{X}^{1/p^{\infty}}]}\otimes _{\mathbb F_p[\underline{X}^{1/p^{\infty}]}}A/p\cong A/p\cdot w(p),\] where $\overline I$ is the image of the ideal $I$ under the canonical surjection $A\twoheadrightarrow A/p$. Since $\mathbb Z_{(p)}[\underline{X}^{1/p^{\infty}}]$ has a delta-structure $\delta$, the Frobenius--Witt differential $w$ can be written as $w(x)=\delta(x)w(p)$. Therefore, the shift of the map $\kappa_A$ can be written as the following:
\[F_{A/p}^*(\overline I/\overline I^2)\to A/p:\overline x\otimes 1\to \overline {\delta(x)},\] where $\overline x$ is an element of $\overline{I}/\overline I^2$, $x\in I$ is a lift of $\overline x$ and $\overline{\delta(x)}$ is the image of $\delta(x)\in \mathbb Z_{(p)}[\underline{X}^{1/p^{\infty}}]$ in $A/p$.

Next, we will compute the map $\lambda$. To do so, we compute the map $\xi$ (Definition \ref{Def5.2}). We have a surjective ring map $\mathbb Z/p^2[\underline{X}^{1/p^{\infty}}]\twoheadrightarrow A/p$ with kernel $(\overline I,p)\subseteq Z/p^2[\underline{X}^{1/p^{\infty}}]$, where $\overline I$ is the image of $I\subseteq Z_{(p)}[\underline{X}^{1/p^{\infty}}]$ in $Z/p^2[\underline{X}^{1/p^{\infty}}]$. Thus we have a canonical equivalence $\pi_1(\mathbb L_{(A/p)/(\mathbb Z/p^2)})\cong (p,\overline I)/(p,\overline I)^2$. Take a surjective ring map $Q\twoheadrightarrow Z/p^2[\underline{X}^{1/p^{\infty}}]$ from a polynomial $\mathbb Z/p^2$-algebra $Q$. Then the composition \[Q\twoheadrightarrow Z/p^2[\underline{X}^{1/p^{\infty}}]\twoheadrightarrow A/p\] is also a surjective ring map. Denote its kernel by $K$. We have a lift \[Q\twoheadrightarrow Z/p^2[\underline{X}^{1/p^{\infty}}]\twoheadrightarrow A/p^2\] of the above surjective ring map along the canonical surjection $A/p^2\twoheadrightarrow A/p$. By construction, this map induces a map \[K/K^2\twoheadrightarrow (p,\overline I)/(p,\overline I)^2\to pA/p^2A:=\operatorname{ker}(A/p^2\twoheadrightarrow A/p).\] By \cite[Tag 0GPU]{Sta26}, the induced map \[(p,\overline I)/(p,\overline I)^2\to pA/p^2A:\overline{px+y}\mapsto p\overline y,\] where $x\in Z/p^2[\underline{X}^{1/p^{\infty}}]$ and $y\in \overline I$, corresponds to the map induced by $\xi$. Therefore, we can write the shift of the map $\xi$ as the composition\[\mathbb L_{(A/p)/(\mathbb Z/p^2)}[-1]\xrightarrow[]{\mathrm{can}}\pi_1(\mathbb L_{(A/p)/(\mathbb Z/p^2)})\cong (p,\overline I)/(p,\overline I)^2\xrightarrow[]{px+y\mapsto y}A/p,\] where $x\in \mathbb Z/p^2[\underline{X}^{1/p^{\infty}}]$ and $y\in \overline I$. Using this computation, we will determine the map $\lambda$ (Definition \ref{Def5.3}). Let $\phi$ be a Frobenius lift on $Z/p^2[\underline{X}^{1/p^{\infty}}]$. Then the map $\pi_1(dF_{A/p}):\pi_1(\mathbb L_{(A/p)/\mathbb Z/p^2})\to \pi_1(\mathbb L_{(A/p)/\mathbb Z/p^2})$ induced by the Frobenius map $F:A/p\to A/p$ can be written as \[(p,\overline I)/(p,\overline I)^2\xrightarrow[]{\overline x\mapsto \overline{\phi(x)}}(p,\overline I)/(p,\overline I)^2.\]
Since we have \begin{align*}
    \overline {\phi(px+y)}&=\overline{p\phi(x)+\phi(y)}\\&=p\overline{x^p}+\overline{y^p}+p\overline {\delta(y)},
\end{align*} where $\delta$ is a delta-structure on $\mathbb Z_{(p)}[\underline{X}^{1/p^{\infty}}]$ compatible with the Frobenius lift on $\mathbb Z/p^2[\underline{X}^{1/p^{\infty}}]$. Therefore the composition $\pi_1(\xi)\circ dF$ (in the sense of Definition \ref{Def5.3}) can be written as the following:
\[(p,\overline I)/(p,\overline I)^2\to A/p:\overline{px+y}\mapsto \overline{x^p}+\overline{\delta(y)},\] where $x\in \mathbb Z_{(p)}[\underline{X}^{1/p^{\infty}}]$ and $y\in I$. On the other hand, the composition $F\circ \pi_1(\xi)$ (in the sense of Definition \ref{Def5.3}) can be written as the following:\[(p,\overline I)/(p,\overline I)^2\to A/p:\overline{px+y}\mapsto \overline{x^p},\] where $x\in \mathbb Z_{(p)}[\underline{X}^{1/p^{\infty}}]$ and $y\in I$. Combining them, the shift of the map $\lambda$ can be written as the following composition:
\[\lambda[-1]:\mathbb L_{(A/p)/(\mathbb Z/p^2)}[-1]\xrightarrow[]{\mathrm{can}}\pi_1(\mathbb L_{(A/p)/(\mathbb Z/p^2)})\cong (p,\overline I)/(p,\overline I)^2\xrightarrow[]{\overline{px+y}\mapsto \overline{\delta(y)}}A/p,\] where $x\in \mathbb Z_{(p)}[\underline{X}^{1/p^{\infty}}]$ and $y\in I$. This map induces \[F_{A/p}^*(\overline I/\overline I^2)\to A/p:\overline y\otimes 1\mapsto \overline{\delta(y)},\] where $y\in I$. Therefore, we get the conclusion.
\end{proof}
The following lemma is used in the above proof.
\begin{lemma}\label{lem5.5}
    The functor\[\operatorname{Flat.Ring}_{\mathbb Z_{(p)}/}\to \mathcal D(\mathbb F_p):A\mapsto F_{A/p}^*\mathbb L_{(A/p)/\mathbb F_p}\] from the category of flat $\mathbb Z_{(p)}$-algebras to the derived $\infty$-category of $\mathbb F_p$-complexes satisfies fpqc descent.
\end{lemma}
\begin{proof}
    For every $A\in \operatorname{Flat.Ring}_{\mathbb Z_{(p)}/}$, there is a natural fiber sequence\[A/p\to F\mathbb L_A\to F_{A/p}^*\mathbb L_{(A/p)/\mathbb F_p}\to.\] Since the functors $A\mapsto A/p(\simeq A/^Lp)$ and $A\mapsto F\mathbb L_{A}$ satisfy fpqc descent (see \cite[Theorem 5.9]{Shi26}), the functor $A\mapsto F^*_{A/p}\mathbb L_{(A/p)/\mathbb F_p}$ also satisfies fpqc descent.
\end{proof}
\begin{definition}\label{def5.6}
    Let $X_1$ be a flat $\mathbb Z/p^2$-scheme. Define the element \[\lambda_{X_1}\in \operatorname{Ext}^1_{X_0}(F_{X_0}^*\mathbb L_{X_0/(\mathbb Z/p^2)},\mathcal O_{X_0})\] by gluing the maps in Definition \ref{Def5.3}.
\end{definition}

\begin{coro}\label{cor5.7}
    Let $X$ be a flat $\mathbb Z_{(p)}$-scheme. Write $X_0:=X\times_{\operatorname{Spec}\mathbb Z_{(p)}}\operatorname{Spec}\mathbb F_p$. Then the arithmetic Kodaira--Spencer class $\kappa_X\in \operatorname{Ext}^1_{X_0}(F^*_{X_0}\mathbb L_{X_0/\mathbb F_p},\mathcal O_{X_0})$ can be characterized by the property that it is targeted by the element $\lambda_{X_1}\in \operatorname{Ext}_{X_0}^1(F^*_{X_0\mathbb L_{X_0/(\mathbb Z/p^2)}},\mathcal O_{X_0})$ (Definition \ref{def5.6}). 
\end{coro}
\begin{proof}
    This follows from Theorem \ref{thm5.4}.
\end{proof}
\begin{coro}\label{cor5.8}
    Let $X$ and $Y$ be flat $\mathbb Z_{(p)}$-schemes. Write $X_1:=X\times _{\operatorname{Spec}\mathbb Z_{(p)}}\operatorname{Spec}\mathbb Z/p^2$ and $Y_1:=Y\times _{\operatorname{Spec}\mathbb Z_{(p)}}\operatorname{Spec}\mathbb Z/p^2$. Assume there is an isomorphism $X_1\cong Y_1$ of schemes. Then there is an equivalence $F\mathbb L_X\simeq F\mathbb L_Y$ under the identification $\mathcal D(X_0)\simeq \mathcal D(Y_0)$.
\end{coro}
\begin{proof}
    This follows from Corollary \ref{cor5.7}.
\end{proof}
\begin{remark}
    Corollary \ref{cor5.8} suggests that the functor \[F\mathbb L_{(-)}:\operatorname{Ani(Ring)}_{\mathbb Z_{(p)}/}\to \mathcal D_{\ge0}(\mathbb F_p)\] of taking Frobenius--Witt cotangent complexes factors through the functor \[(-)\otimes^L_{\mathbb Z_{(p)}}\mathbb Z/p^2:\operatorname{Ani(Ring)}_{\mathbb Z_{(p)}/}\to \operatorname{Ani(Ring)}_{(\mathbb Z/p^2)/}\] of taking derived modulo $p^2$. This can indeed be proved using a general categorical argument. 

    By \cite[Corollary 2.4]{Sai22}, we have the following commutative diagram:

    \[ \begin{tikzcd} \operatorname{Ring}_{\mathbb Z_{(p)}/} \arrow[r, "F\Omega_{(-)}"] \arrow[d, "(-)\otimes_{\mathbb Z_{(p)}}\mathbb Z/p^2"'] & \operatorname{Mod}(\mathbb F_p) \\ \operatorname{Ring}_{(\mathbb Z/p^2)/}
    \arrow[ur, "F\Omega_{(-)}"']
    \end{tikzcd} \] 
(note that the diagonal functor coincides with a special case of the functor $\Omega_{(-)}^{\mathrm{tot}}$ defined in \cite[Paragraph 2.7]{DKRZB17}). Since the vertical functor preserves the set of compact projective generators and its animation is the functor \[(-)\otimes^L_{\mathbb Z_{(p)}}\mathbb Z/p^2:\operatorname{Ani(Ring)}_{\mathbb Z_{(p)}/}\to \operatorname{Ani(Ring)}_{(\mathbb Z/p^2)/}\] of taking derived modulo $p^2$, \cite[Proposition 5.15]{CS24} implies that the following diagram is commutative:

     \[ \begin{tikzcd} \operatorname{Ani(Ring)}_{\mathbb Z_{(p)}/} \arrow[r, "F\mathbb L_{(-)}"] \arrow[d, "(-)\otimes_{\mathbb Z_{(p)}}^L\mathbb Z/p^2"'] & \mathcal D_{\ge0}(\mathbb F_p) \\ \operatorname{Ani(Ring)}_{(\mathbb Z/p^2)/}
    \arrow[ur, "\mathbb L^{\mathrm{tot}}_{(-)}"']
    \end{tikzcd} \] 
where $\mathbb L^{\mathrm{tot}}_{(-)}$ is the animation of the functor $\Omega^{\mathrm{tot}}_{(-)}$.
    
\end{remark}


\begin{thebibliography}{DKRZB17}
\bibitem[AWZ21]{AWZ21}
P.~Achinger, J.~Witaszek and M.~E.~Zdanowicz, 
\newblock{\em Global Frobenius liftability I}.
\newblock J. Eur. Math. Soc. (JEMS) {\bf 23} (2021), no.~8, 2601--2648; MR4269423.

\bibitem[Bha12]{Bha12}
B.~Bhatt.
\newblock{\em p-adic derived de Rham cohomology}.
\newblock arXiv:1204.6560.
\bibitem[BTLM97]{BTLM97}
A.~Buch, J.F.~Thomsen, N.~Lauritzen, and V.~Mehta.
\newblock{\em The Frobenius morphism on a
toric variety}.
\newblock Tohoku Math. J. (2) 49 (1997), no. 3, 355–366. MR 1464183.

\bibitem[ČS24]{CS24}
K.~Česnavičius, P.~Scholze.
\newblock{\em Purity for flat cohomology}.
\newblock Ann. of Math. (2) 199 (2024), no. 1, 51--180. MR4681144

\bibitem[DI87]{DI87}
 P.~Deligne and L.~Illusie.
 \newblock{\em Relèvements modulo $p^2$ et décomposition du complexe de de Rham}.
 \newblock {Invent. Math. 89 (1987), no. 2, 247–270}.
 
\bibitem[DKRZB17]{DKRZB17}
T.~Dupuy, E.~Katz, J.~Rabinoff, and D.~Zureick-Brown.
\newblock{\em Total $p$-differentials on schemes
over $\mathbb Z/p^2$}.

\bibitem[Ill71]{Ill71}
L.~Illusie.
\newblock{\em Complexe cotangent et d\'eformations. I}. \newblock Lecture Notes in Mathematics, Vol. 239, Springer, Berlin-New York, 1971; MR0491680.
\bibitem[Ill72]{Ill72}
L.~Illusie.
\newblock{\em Complexe cotangent et d\'eformations. II}.
\newblock Lecture Notes in Mathematics, Vol. 283, Springer, Berlin-New York, 1972; MR0491681.

\bibitem[Iye07]{Iye07}
S.~B. Iyengar.
\newblock{\em Andr\'e-Quillen homology of commutative algebras}.
\newblock {\it Interactions between homotopy theory and algebra} (2007), 203--234, Contemp. Math. 436, Amer. Math. Soc. Providence, RI,  MR2355775

\bibitem[Lur17]{Lur17}
J.~Lurie.
\newblock{\em Higher algebra}.
\newblock
{\url{https://people.math.harvard.edu/~lurie/papers/HA.pdf}, sep 2017}
\bibitem[Mao26]{Mao26}
Z.~Mao.
\newblock{Frobenius--Witt cotangent complex for derived rings}.
\newblock  arXiv:2608.04487v1.

\bibitem[Sai22]{Sai22}
T.~Saito.
\newblock{\em Frobenius--Witt differentials and regularity}.
\newblock Algebra Number Theory 16 (2022), no. 2, 369–391. MR4412577
\bibitem[Shi26]{Shi26}
K.~Shimada.
\newblock{\em Frobenius--Witt cotangent complexes}.
\newblock arXiv:2605.14803. 
\bibitem[Sta26]{Sta26}
The Stacks Project Authors.
\newblock{\em The Stacks Project}.
\newblock{\url{https://stacks.math.columbia.edu/}}.
\bibitem[Zda18]{Zda18}
M.~E.~Zdanowicz.
\newblock{\em Liftability of singularities and their Frobenius morphism modulo $p^2$}.
\newblock Int. Math. Res. Not. IMRN {\bf 2018}, no.~14, 4513--4577; MR3830576.
\end{thebibliography}
\end{document}